\documentclass{amsart}
\usepackage{amsmath}
\usepackage{amssymb}
\usepackage{amsfonts}
\usepackage{graphicx}
\usepackage{tikz}
\usetikzlibrary{decorations.markings}

\usepackage{longtable}
\usepackage{array}

\usepackage{lipsum} 
\usepackage[T1]{fontenc}
\newtheorem{theorem}{Theorem}

\tikzset{line_diagram/.style={draw=black, thick, line cap=round}}
\tikzset{grey_line_diagram/.style={draw=gray, line cap=round}}
\tikzset{orientation_line/.style={draw=gray, thin, ->}}
\tikzset{
  line_arrow/.style={
        draw=black,
        thick,
        line cap=round,
        postaction={decorate,decoration={
            markings,
            mark=at position .5 with {\arrow[#1]{>}}
        }}},
}
\def\PointRadius{0.03}
\newcommand{\K}{\mathbb{K}}

\title{Pentagon relation and Biedenharn -- Elliott identity}
\author{Korablev Ph. G.}

\address{Chelyabinsk State University, Chelyabinsk, Russia; N.N. Krasovsky Institute of Mathematics and Meckhanics, Ekaterinburg, Russia}
\email{korablev@csu.ru}
\date{}

\begin{document}

\begin{abstract}
    The main subject of the paper is the pentagon relation. This relation can be expressed in different ways. We start with the natural geometric form of the pentagon relation. Then we express it in algebraic form as a family of equations with a set of linear maps as variables. Next, we derive several equivalent forms of the algebraic pentagon relation. These forms can be expressed using the classical notion of 6j-symbols. Finally, we show how to extract a solution of the pentagon relation from any modular category.
\end{abstract}

\maketitle

\section{Geometric form of the pentagon relation}

There are five different triangulations of the pentagon. It's possible to change any of these triangulations into another one by using diagonal flips. On the figure \ref{Figure:Pentagon}, the most left triangulation can be transformed into the most right one in two different ways: by using three flips (upper part of the figure) or by using two flips (lower part of the figure). This diagram is a geometric form of the pentagon relation.

\begin{figure}[h]
    \begin{tikzpicture}
        \draw (3, 0) node {
            \begin{tikzpicture}
                \draw[line_diagram] (0.7, 1) -- (0, -1);
                \draw[line_diagram] (0.7, 1) -- (-1, 0);
                \filldraw (1, 0) circle (\PointRadius);
                \filldraw (-1, 0) circle (\PointRadius);
                \filldraw (0.7, 1) circle (\PointRadius);
                \filldraw (-0.7, 1) circle (\PointRadius);
                \filldraw (0, -1) circle (\PointRadius);
                \draw[line_diagram] (-0.7, 1) -- (-1, 0);
                \draw[line_diagram] (-1, 0) -- (0, -1);
                \draw[line_diagram] (0, -1) -- (1, 0);
                \draw[line_diagram] (1, 0) -- (0.7, 1);
                \draw[line_diagram] (-0.7, 1) -- (0.7, 1);
            \end{tikzpicture}
        };
        \draw (-3, 0) node {
            \begin{tikzpicture}
                \draw[line_diagram] (-0.7, 1) -- (0, -1);
                \draw[line_diagram] (-0.7, 1) -- (1, 0);
                \filldraw (1, 0) circle (\PointRadius);
                \filldraw (-1, 0) circle (\PointRadius);
                \filldraw (0.7, 1) circle (\PointRadius);
                \filldraw (-0.7, 1) circle (\PointRadius);
                \filldraw (0, -1) circle (\PointRadius);
                \draw[line_diagram] (-0.7, 1) -- (-1, 0);
                \draw[line_diagram] (-1, 0) -- (0, -1);
                \draw[line_diagram] (0, -1) -- (1, 0);
                \draw[line_diagram] (1, 0) -- (0.7, 1);
                \draw[line_diagram] (-0.7, 1) -- (0.7, 1);
            \end{tikzpicture}
        };
        \draw (2, 3.1) node {
            \begin{tikzpicture}
                \draw[line_diagram] (0.7, 1) -- (-1, 0);
                \draw[line_diagram] (-1, 0) -- (1, 0);
                \filldraw (1, 0) circle (\PointRadius);
                \filldraw (-1, 0) circle (\PointRadius);
                \filldraw (0.7, 1) circle (\PointRadius);
                \filldraw (-0.7, 1) circle (\PointRadius);
                \filldraw (0, -1) circle (\PointRadius);
                \draw[line_diagram] (-0.7, 1) -- (-1, 0);
                \draw[line_diagram] (-1, 0) -- (0, -1);
                \draw[line_diagram] (0, -1) -- (1, 0);
                \draw[line_diagram] (1, 0) -- (0.7, 1);
                \draw[line_diagram] (-0.7, 1) -- (0.7, 1);
            \end{tikzpicture}
        };
        \draw (-2, 3.1) node {
            \begin{tikzpicture}
                \draw[line_diagram] (-0.7, 1) -- (1, 0);
                \draw[line_diagram] (-1, 0) -- (1, 0);
                \filldraw (1, 0) circle (\PointRadius);
                \filldraw (-1, 0) circle (\PointRadius);
                \filldraw (0.7, 1) circle (\PointRadius);
                \filldraw (-0.7, 1) circle (\PointRadius);
                \filldraw (0, -1) circle (\PointRadius);
                \draw[line_diagram] (-0.7, 1) -- (-1, 0);
                \draw[line_diagram] (-1, 0) -- (0, -1);
                \draw[line_diagram] (0, -1) -- (1, 0);
                \draw[line_diagram] (1, 0) -- (0.7, 1);
                \draw[line_diagram] (-0.7, 1) -- (0.7, 1);
            \end{tikzpicture}
        };
        \draw (0, -2.2) node {
            \begin{tikzpicture}
                \draw[line_diagram] (-0.7, 1) -- (0, -1);
                \draw[line_diagram] (0.7, 1) -- (0, -1);
                \filldraw (1, 0) circle (\PointRadius);
                \filldraw (-1, 0) circle (\PointRadius);
                \filldraw (0.7, 1) circle (\PointRadius);
                \filldraw (-0.7, 1) circle (\PointRadius);
                \filldraw (0, -1) circle (\PointRadius);
                \draw[line_diagram] (-0.7, 1) -- (-1, 0);
                \draw[line_diagram] (-1, 0) -- (0, -1);
                \draw[line_diagram] (0, -1) -- (1, 0);
                \draw[line_diagram] (1, 0) -- (0.7, 1);
                \draw[line_diagram] (-0.7, 1) -- (0.7, 1);
            \end{tikzpicture}
        };
        \draw[->] (-3, 1.2) -- (-2.5, 2.2);
        \draw[->] (2.5, 2.2) -- (3, 1.2);
        \draw[->] (-0.8, 3.5) -- (0.8, 3.5);
        \draw[->] (-2.5, -0.8) -- (-1.1, -1.7);
        \draw[->] (1.1, -1.7) -- (2.5, -0.8);
    \end{tikzpicture}
    \caption{\label{Figure:Pentagon}Geometric form of the pentagon relation}
\end{figure}

\section{Algebraic form of the pentagon relation}

In this section we will translate the geometric form of the pentagon relation into algebraic language. We will primarily use the approach of \cite{Ya}.

Let $I$ be a finite set. The elements of this set will play the role of colours for edge colouring and indices for module indexing. Fix the family of modules $V_{ab}^c$ over the ring $\K$, defined for each $a, b, c\in I$. The algebraic form of the pentagon relation is an equation for linear maps over these modules.

Consider a triangle with vertices ordered by the symbols $0, 1, 2$. This order defines the orientation of the triangle and the orientation of each edge: the edge oriented from vertex $i$ to vertex $j$ if $i < j$. It's clear that induced edge orientations coincide for two edges and opposite for the third one. Colour the edges of the triangle using elements from set the $I$. Then this coloured triangle corresponds to the module $V_{ab}^c$, where $a$ is the colour of the edge $[0, 1]$, $b$ is the colour of the edge $[1, 2]$ and $c$ is the colour of the edge $[0, 2]$ (figure \ref{Figure:ColouredTriangle}).

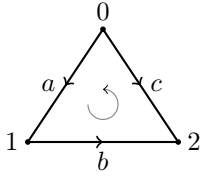
\begin{figure}[h]
    \begin{tikzpicture}
        \filldraw (0, 1.5) circle (\PointRadius) node[above] {0};
        \filldraw (1, 0) circle (\PointRadius) node[right] {2};
        \filldraw (-1, 0) circle (\PointRadius) node[left] {1};
        \draw[line_arrow] (-1, 0) -- (1, 0) node[below, pos=0.5] {$b$};
        \draw[line_arrow] (0, 1.5) -- (1, 0) node[right, pos=0.5] {$c$};
        \draw[line_arrow] (0, 1.5) -- (-1, 0) node[left, pos=0.5] {$a$};
        \draw[orientation_line] (-0.2, 0.5) arc (-180:90:0.2);
    \end{tikzpicture}
    \caption{\label{Figure:ColouredTriangle}Coloured oriented triangle corresponds to the module $V_{ab}^c$}
\end{figure}

Consider the triangulated square with ordered vertices (figure \ref{Figure:ColouredSquare} on the left). As before, orient edges from the vertex with the smaller value to the vertex with the larger value. Colour the boundary edges with colours $a, b, c, d\in I$ and the inner diagonal with the colour $x\in I$. This coloured square corresponds to the tensor product of the modules, which corresponds to the triangles of the triangulation: $V_{xc}^d\otimes V_{ab}^x$.

\begin{figure}[h]
    \begin{tikzpicture}
        \filldraw (0, 0) circle (\PointRadius) node[below left] {1};
        \filldraw (1.5, 0) circle (\PointRadius) node[below right] {2};
        \filldraw (1.5, 1.5) circle (\PointRadius) node[above right] {3};
        \filldraw (0, 1.5) circle (\PointRadius) node[above left] {0};
        \draw[line_arrow] (0, 0) -- (1.5, 0) node[below, pos=0.5] {$b$};
        \draw[line_arrow] (1.5, 0) -- (1.5, 1.5) node[right, pos=0.5] {$c$};
        \draw[line_arrow] (0, 1.5) -- (0, 0) node[left, pos=0.5] {$a$};
        \draw[line_arrow] (0, 1.5) -- (1.5, 1.5) node[above, pos=0.5] {$d$};
        \draw[line_arrow] (0, 1.5) -- (1.5, 0) node[right, pos=0.5] {$x$};
    \end{tikzpicture}
    \hspace{1cm}
    \begin{tikzpicture}
        \filldraw (0, 0) circle (\PointRadius) node[below left] {1};
        \filldraw (1.5, 0) circle (\PointRadius) node[below right] {2};
        \filldraw (1.5, 1.5) circle (\PointRadius) node[above right] {3};
        \filldraw (0, 1.5) circle (\PointRadius) node[above left] {0};
        \draw[line_arrow] (0, 0) -- (1.5, 0) node[below, pos=0.5] {$b$};
        \draw[line_arrow] (1.5, 0) -- (1.5, 1.5) node[right, pos=0.5] {$c$};
        \draw[line_arrow] (0, 1.5) -- (0, 0) node[left, pos=0.5] {$a$};
        \draw[line_arrow] (0, 1.5) -- (1.5, 1.5) node[above, pos=0.5] {$d$};
        \draw[line_arrow] (0, 0) -- (1.5, 1.5) node[right, pos=0.5] {$y$};
    \end{tikzpicture}
    \caption{\label{Figure:ColouredSquare}The coloured square corresponds to the tensor product $V_{xc}^d\otimes V_{ab}^x$ (on the left) and to the tensor product $V_{ay}^d\otimes V_{bc}^y$ (on the right).}
\end{figure}
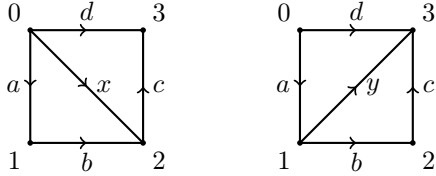

There is another triangulation of the oriented square (figure \ref{Figure:ColouredSquare} on the right). It corresponds to the tensor product $V_{ay}^d\otimes V_{bc}^y$. For both triangulations we order the triangles by using the same rule: the first triangle in the triangulation is the triangle incident to the boundary edge $[0, 3]$.

The triangulated square with coloured boundary edges (and non-coloured diagonal) corresponds to the direct sum of the modules over all possible colours of the inner diagonal. We draw this direct sum as a square with a grey diagonal. So the square on the left of the figure \ref{Figure:SquareAnyDiagonal} corresponds to the module $$\bigoplus_{x\in I}V_{xc}^d\otimes V_{ab}^x,$$ and the square on the right of the figure \ref{Figure:SquareAnyDiagonal} corresponds to the module $$\bigoplus_{y\in I}V_{ay}^d\otimes V_{bc}^y.$$.

\begin{figure}[h]
    \begin{tikzpicture}
        \draw[grey_line_diagram] (0, 1.5) -- (1.5, 0);
        \filldraw (0, 0) circle (\PointRadius) node[below left] {1};
        \filldraw (1.5, 0) circle (\PointRadius) node[below right] {2};
        \filldraw (1.5, 1.5) circle (\PointRadius) node[above right] {3};
        \filldraw (0, 1.5) circle (\PointRadius) node[above left] {0};
        \draw[line_arrow] (0, 0) -- (1.5, 0) node[below, pos=0.5] {$b$};
        \draw[line_arrow] (1.5, 0) -- (1.5, 1.5) node[right, pos=0.5] {$c$};
        \draw[line_arrow] (0, 1.5) -- (0, 0) node[left, pos=0.5] {$a$};
        \draw[line_arrow] (0, 1.5) -- (1.5, 1.5) node[above, pos=0.5] {$d$};
    \end{tikzpicture}
    \hspace{1cm}
    \begin{tikzpicture}
        \draw[grey_line_diagram] (0, 0) -- (1.5, 1.5);
        \filldraw (0, 0) circle (\PointRadius) node[below left] {1};
        \filldraw (1.5, 0) circle (\PointRadius) node[below right] {2};
        \filldraw (1.5, 1.5) circle (\PointRadius) node[above right] {3};
        \filldraw (0, 1.5) circle (\PointRadius) node[above left] {0};
        \draw[line_arrow] (0, 0) -- (1.5, 0) node[below, pos=0.5] {$b$};
        \draw[line_arrow] (1.5, 0) -- (1.5, 1.5) node[right, pos=0.5] {$c$};
        \draw[line_arrow] (0, 1.5) -- (0, 0) node[left, pos=0.5] {$a$};
        \draw[line_arrow] (0, 1.5) -- (1.5, 1.5) node[above, pos=0.5] {$d$};
    \end{tikzpicture}
    \caption{\label{Figure:SquareAnyDiagonal}The square on the left corresponds to the module $\bigoplus\limits_{x\in I}V_{xc}^d\otimes V_{ab}^x$, and the square on the right corresponds to the module $\bigoplus\limits_{y\in I}V_{ay}^d\otimes V_{bc}^y$}
\end{figure}
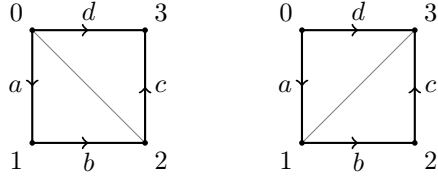

Let $$\mathcal{F}_{abc}^{d}\colon \bigoplus_{x\in I}V_{xc}^d\otimes V_{ab}^x\to \bigoplus_{y\in I}V_{ay}^d\otimes V_{bc}^y$$ be the linear map between two modules, corresponding to different triangulations of the coloured square. This map corresponds to the flip of the diagonal. That's why it is denoted by $\mathcal{F}$. Let $$\mathcal{F}_{abc}^{d}|^{x}_{y}\colon V_{xc}^d\otimes V_{ab}^x \to V_{ay}^d\otimes V_{bc}^y$$ denote components of the map $\mathcal{F}_{abc}^d$ between summands of the input and output modules.

For each $\alpha\in \bigoplus\limits_{x\in I}V_{xc}^d\otimes V_{ab}^x$ denote by $\alpha^{x}\in V_{xc}^d\otimes V_{ab}^x$ the component of the vector $\alpha$ from the summand that corresponds to the colour $x\in I$. It's clear that $$\alpha = \sum\limits_{x\in I}\alpha^x.$$ Then $$\mathcal{F}_{abc}^d(\alpha) = \mathcal{F}_{abc}^{d}\left(\sum\limits_{x\in I}\alpha^x\right) = \sum\limits_{x, y\in I}\mathcal{F}_{abc}^{d}|^{x}_{y}(\alpha^x).$$

For any $a, b, c\in I$ denote $id_{ab}^c\colon V_{ab}^c\to V_{ab}^c$ the identity map.

Consider the pentagon. As before, define the cyclic order of its vertices, and the colouring of its edges. There are five different triangulations of the pentagon. On the figure \ref{Figure:PentagonTriangulations} all these triangulations are shown with the corresponding modules.

\begin{figure}[h]
    \begin{tikzpicture}
        \draw[grey_line_diagram] (-0.7, 1) -- (0, -1);
        \draw[grey_line_diagram] (-0.7, 1) -- (1, 0);
        \filldraw (1, 0) circle (\PointRadius);
        \filldraw (-1, 0) circle (\PointRadius);
        \filldraw (0.7, 1) circle (\PointRadius);
        \filldraw (-0.7, 1) circle (\PointRadius);
        \filldraw (0, -1) circle (\PointRadius);
        \draw[line_arrow] (-0.7, 1) -- (-1, 0) node[left, pos=0.5] {$a$};
        \draw[line_arrow] (-1, 0) -- (0, -1) node[left, pos=0.5] {$b$};
        \draw[line_arrow] (0, -1) -- (1, 0) node[right, pos=0.5] {$c$};
        \draw[line_arrow] (1, 0) -- (0.7, 1) node[right, pos=0.5] {$d$};
        \draw[line_arrow] (-0.7, 1) -- (0.7, 1) node[above, pos=0.5] {$e$};
        \draw (0, -1.75) node {$\bigoplus\limits_{x, y\in I}V_{yd}^e\otimes V_{xc}^y\otimes V_{ab}^x$};
    \end{tikzpicture}
    \hfill
    \begin{tikzpicture}
        \draw[grey_line_diagram] (-1, 0) -- (1, 0);
        \draw[grey_line_diagram] (-0.7, 1) -- (1, 0);
        \filldraw (1, 0) circle (\PointRadius);
        \filldraw (-1, 0) circle (\PointRadius);
        \filldraw (0.7, 1) circle (\PointRadius);
        \filldraw (-0.7, 1) circle (\PointRadius);
        \filldraw (0, -1) circle (\PointRadius);
        \draw[line_arrow] (-0.7, 1) -- (-1, 0) node[left, pos=0.5] {$a$};
        \draw[line_arrow] (-1, 0) -- (0, -1) node[left, pos=0.5] {$b$};
        \draw[line_arrow] (0, -1) -- (1, 0) node[right, pos=0.5] {$c$};
        \draw[line_arrow] (1, 0) -- (0.7, 1) node[right, pos=0.5] {$d$};
        \draw[line_arrow] (-0.7, 1) -- (0.7, 1) node[above, pos=0.5] {$e$};
        \draw (0, -1.75) node {$\bigoplus\limits_{y, z\in I}V_{yd}^e\otimes V_{az}^y\otimes V_{bc}^z$};
    \end{tikzpicture}
    \hfill
    \begin{tikzpicture}
        \draw[grey_line_diagram] (-1, 0) -- (1, 0);
        \draw[grey_line_diagram] (-1, 0) -- (0.7, 1);
        \filldraw (1, 0) circle (\PointRadius);
        \filldraw (-1, 0) circle (\PointRadius);
        \filldraw (0.7, 1) circle (\PointRadius);
        \filldraw (-0.7, 1) circle (\PointRadius);
        \filldraw (0, -1) circle (\PointRadius);
        \draw[line_arrow] (-0.7, 1) -- (-1, 0) node[left, pos=0.5] {$a$};
        \draw[line_arrow] (-1, 0) -- (0, -1) node[left, pos=0.5] {$b$};
        \draw[line_arrow] (0, -1) -- (1, 0) node[right, pos=0.5] {$c$};
        \draw[line_arrow] (1, 0) -- (0.7, 1) node[right, pos=0.5] {$d$};
        \draw[line_arrow] (-0.7, 1) -- (0.7, 1) node[above, pos=0.5] {$e$};
        \draw (0, -1.75) node {{$\bigoplus\limits_{z, p\in I}V_{ap}^e\otimes V_{zd}^p\otimes V_{bc}^z$}};
    \end{tikzpicture}

    \hspace{1.5cm}
    \begin{tikzpicture}
        \draw[grey_line_diagram] (0.7, 1) -- (0, -1);
        \draw[grey_line_diagram] (-1, 0) -- (0.7, 1);
        \filldraw (1, 0) circle (\PointRadius);
        \filldraw (-1, 0) circle (\PointRadius);
        \filldraw (0.7, 1) circle (\PointRadius);
        \filldraw (-0.7, 1) circle (\PointRadius);
        \filldraw (0, -1) circle (\PointRadius);
        \draw[line_arrow] (-0.7, 1) -- (-1, 0) node[left, pos=0.5] {$a$};
        \draw[line_arrow] (-1, 0) -- (0, -1) node[left, pos=0.5] {$b$};
        \draw[line_arrow] (0, -1) -- (1, 0) node[right, pos=0.5] {$c$};
        \draw[line_arrow] (1, 0) -- (0.7, 1) node[right, pos=0.5] {$d$};
        \draw[line_arrow] (-0.7, 1) -- (0.7, 1) node[above, pos=0.5] {$e$};
        \draw (0, -1.75) node {{$\bigoplus\limits_{p, q\in I}V_{ap}^e\otimes V_{bq}^p\otimes V_{cd}^q$}};
    \end{tikzpicture}
    \hfill
    \begin{tikzpicture}
        \draw[grey_line_diagram] (0.7, 1) -- (0, -1);
        \draw[grey_line_diagram] (-0.7, 1) -- (0, -1);
        \filldraw (1, 0) circle (\PointRadius);
        \filldraw (-1, 0) circle (\PointRadius);
        \filldraw (0.7, 1) circle (\PointRadius);
        \filldraw (-0.7, 1) circle (\PointRadius);
        \filldraw (0, -1) circle (\PointRadius);
        \draw[line_arrow] (-0.7, 1) -- (-1, 0) node[left, pos=0.5] {$a$};
        \draw[line_arrow] (-1, 0) -- (0, -1) node[left, pos=0.5] {$b$};
        \draw[line_arrow] (0, -1) -- (1, 0) node[right, pos=0.5] {$c$};
        \draw[line_arrow] (1, 0) -- (0.7, 1) node[right, pos=0.5] {$d$};
        \draw[line_arrow] (-0.7, 1) -- (0.7, 1) node[above, pos=0.5] {$e$};
        \draw (0, -1.75) node {{$\bigoplus\limits_{x, q\in I}V_{xq}^e\otimes V_{cd}^q\otimes V_{ab}^x$}};
    \end{tikzpicture}
    \hspace{1.5cm}
    \caption{\label{Figure:PentagonTriangulations}Different triangulations of the coloured pentagon and corresponding modules}
\end{figure}
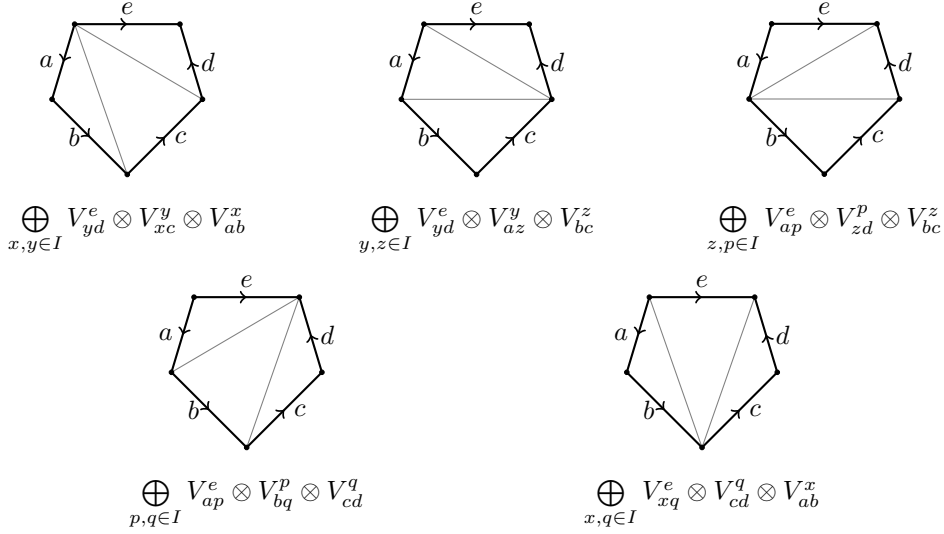

Let $\bigoplus\limits_{x}U_x$ and $\bigoplus\limits_{x}W_x$ be two modules with summands indexed by some $x$, and assume that for each $x$ there exists a linear map $f_x\colon U_x \to V_x$. Then denote $$\bigoplus_x f_x\colon \bigoplus\limits_{x}U_x \to \bigoplus\limits_{x}W_x$$ the corresponding linear map which acts on each summand $U_x$ as $f_x$. Using these notations, we can write the map between modules corresponding to different triangulations of the pentagon.

The map between the first and the second triangulation from the figure \ref{Figure:PentagonTriangulations} is $\bigoplus\limits_{y\in I}id_{yd}^e\otimes\mathcal{F}_{abc}^{y}$, the map between the second and the third triangulation is $\bigoplus\limits_{z\in I}\mathcal{F}_{azd}^{e}\otimes id_{bc}^{z}$, and the map between the third and the forth triangulation is $\bigoplus\limits_{p\in I}id_{ap}^{e}\otimes \mathcal{F}_{pbc}^{d}$. In the similar way the map between the first and the fifth triangulation is $\bigoplus\limits_{x\in I}\mathcal{F}_{xcd}^{e}\otimes id_{ab}^{x}$, and the map between the fifth and the forth triangulation is a composition of the permutation $$P_{23}\colon \bigoplus\limits_{x, q\in I}V_{xq}^e\otimes V_{cd}^q\otimes V_{ab}^x\to \bigoplus\limits_{x, q\in I}V_{xq}^e\otimes V_{ab}^x\otimes V_{cd}^q,$$ which permute the second and the third multipliers in each summand, and the map $\bigoplus\limits_{q\in I}\mathcal{F}_{abq}^{e}\otimes id_{cd}^{q}$.

Consequently, the pentagon relation can be written as follows

\begin{multline*}
    \Bigl(\bigoplus\limits_{y\in I}id_{yd}^e\otimes\mathcal{F}_{abc}^{y}\Bigr)\circ \Bigl(\bigoplus\limits_{z\in I}\mathcal{F}_{azd}^{e}\otimes id_{bc}^{z}\Bigr)\circ \Bigl(\bigoplus\limits_{p\in I}id_{ap}^{e}\otimes \mathcal{F}_{bcd}^{p}\Bigr) = \\ = \Bigl(\bigoplus\limits_{x\in I}\mathcal{F}_{xcd}^{e}\otimes id_{ab}^{x}\Bigr)\circ P_{23}\circ \Bigl(\bigoplus\limits_{q\in I}\mathcal{F}_{abq}^{e}\otimes id_{cd}^{q}\Bigr).
\end{multline*}

In this equation we write a composition of linear maps from left to right. The left part of the equation is a composition of maps
\begin{multline*}
    \bigoplus\limits_{x, y\in I}V_{yd}^e\otimes V_{xc}^y\otimes V_{ab}^x \to \bigoplus\limits_{y, z\in I}V_{yd}^e\otimes V_{az}^y\otimes V_{bc}^z \to \\ \to \bigoplus\limits_{z, p\in I}V_{ap}^e\otimes V_{zd}^p\otimes V_{bc}^z \to \bigoplus\limits_{p, q\in I}V_{ap}^e\otimes V_{bq}^p\otimes V_{cd}^q,
\end{multline*}
and the right part is a composition of maps
\begin{multline*}
    \bigoplus\limits_{x, y\in I}V_{yd}^e\otimes V_{xc}^y\otimes V_{ab}^x \to \bigoplus\limits_{x, q\in I}V_{xq}^e\otimes V_{cd}^q\otimes V_{ab}^x \to \\ \to \bigoplus\limits_{x, q\in I}V_{xq}^e\otimes V_{ab}^x\otimes V_{cd}^q \to \bigoplus\limits_{p, q\in I}V_{ap}^e\otimes V_{bq}^p\otimes V_{cd}^q.
\end{multline*}

In fact, the pentagon relation defines the system of equations that can be written for any $a, b, c, d, e\in I$. The solution of the pentagon relation is a family of linear maps $\{\mathcal{F}_{abc}^{d}\}$, defined for any $a, b, c, d\in I$, which satisfy all these equations.

\section{Three forms of the pentagon relation}

In this section we derive several alternative forms of the pentagon relation. They are all equivalent to the original one, but expressed in slightly different terms.

\subsection{Component form}

\begin{theorem}
    \label{Theorem:ShortForm}
    Let the family $\{\mathcal{F}_{abc}^{d}\}$ be a solution of the pentagon relation. Then for each $a, b, c, d, e, x, y, p, q\in I$ the components $\{\mathcal{F}_{abc}^{d}|^{x}_{y}\}$ satisfy to the following equation:
    \begin{multline*}
        \sum\limits_{z\in I}(id_{yd}^e\otimes \mathcal{F}_{abc}^{y}|^{x}_{z})\circ (\mathcal{F}_{azd}^{e}|^{y}_{p}\otimes id_{bc}^{z})\circ (id_{ap}^{e}\otimes \mathcal{F}_{bcd}^{p}|^{z}_{q}) = \\ = (\mathcal{F}_{xcd}^{e}|^{y}_{q}\otimes id_{ab}^{x})\circ P_{23}\circ (\mathcal{F}_{abq}^{e}|^{x}_{p}\otimes id_{cd}^{q}).
    \end{multline*}
\end{theorem}
\begin{proof}
    Denote the left hand side of the pentagon relation by $\mathcal{L}$, and the right hand side by $\mathcal{R}$. 
    
    Let $$\alpha\in \bigoplus\limits_{x, y\in I}V_{yd}^e\otimes V_{xc}^y\otimes V_{ab}^x$$ and let $$\alpha^y\in \bigoplus\limits_{x\in I} V_{yd}^e\otimes V_{xc}^y\otimes V_{ab}^x$$ be the component of the vector $\alpha$, which corresponds to summands with fixed $y\in I$. Similarly, $$\alpha^{xy}\in V_{yd}^e\otimes V_{xc}^y\otimes V_{ab}^x$$ is the component corresponding to the summand with both $x, y\in I$ fixed.
    
    Calculate 
    \begin{multline*}
        \Bigl(\bigoplus_{y\in I} id_{yd}^{e}\otimes \mathcal{F}_{abc}^{y}\Bigr)(\alpha) = \Bigl(\bigoplus_{y\in I} id_{yd}^{e}\otimes \mathcal{F}_{abc}^{y}\Bigr)\Bigl(\sum\limits_{y\in I}\alpha^y\Bigr) = \\ = \sum\limits_{y\in I} (id_{yd}^e\otimes \mathcal{F}_{abc}^{y}) \Bigl(\sum\limits_{x\in I}\alpha^{xy}\Bigr) = \sum\limits_{x, y, z\in I}(id_{yd}^{e}\otimes \mathcal{F}_{abc}^{y}|^{x}_{z})(\alpha^{xy}).
    \end{multline*}

    Denote the obtained vector by $$\beta\in \bigoplus\limits_{y, z\in I}V_{yd}^e\otimes V_{az}^y\otimes V_{bc}^z.$$

    By using similar notations calculate $$\Bigl(\bigoplus_{z\in I}\mathcal{F}_{azd}^{e}\otimes id_{bc}^{z}\Bigr)(\beta) = \sum\limits_{p, y, z\in I}(\mathcal{F}_{azd}^{e}|^{y}_{p}\otimes id_{bc}^{z})(\beta^{yz}).$$

    Denote the obtained vector by $$\gamma\in \bigoplus\limits_{z, p\in I}V_{ap}^e\otimes V_{zd}^p\otimes V_{bc}^z.$$

    Next, calculate $$\Bigl(\bigoplus_{p\in I}id_{ap}^e\otimes \mathcal{F}_{bcd}^{p}\Bigr)(\gamma) = \sum\limits_{p, q, z\in I}(id_{ap}^{e}\otimes \mathcal{F}_{bcd}^{p}|^{z}_{q})(\gamma^{pz}).$$

    The result is the vector $$\mathcal{L}(\alpha)\in \bigoplus_{p, q\in I}V_{ap}^{e}\otimes V_{bq}^{p}\otimes V_{cd}^{q}.$$

    Then 
    \begin{multline*}
        \mathcal{L}(\alpha)^{pq} = \sum\limits_{z\in I}(id_{ap}^{e}\otimes \mathcal{F}_{bcd}^{p}|^{z}_{q})(\gamma^{pz}) = \\ = \sum\limits_{y, z\in I}\left((\mathcal{F}_{azd}^{e}|^{y}_{p}\otimes id_{bc}^{z})\circ (id_{ap}^{e}\otimes \mathcal{F}_{bcd}^{p}|^{z}_{q})\right) (\beta^{yz}) = \\ = \sum\limits_{x, y, z\in I}\left((id_{yd}^e\otimes \mathcal{F}_{abc}^{y}|^{x}_{z})\circ (\mathcal{F}_{azd}^{e}|^{y}_{p}\otimes id_{bc}^{z})\circ (id_{ap}^{e}\otimes \mathcal{F}_{bcd}^{p}|^{z}_{q})\right) (\alpha^{xy}).
    \end{multline*}

    In the same way $$\mathcal{R}(\alpha)^{pq} = \sum\limits_{x, y\in I}\left((\mathcal{F}_{xcd}^{e}|^{y}_{q}\otimes id_{ab}^{x})\circ P_{23}\circ (\mathcal{F}_{abq}^{e}|^{x}_{p}\otimes id_{cd}^{q})\right)(\alpha^{xy}).$$

    The theorem statement follows from the fact that $\mathcal{L}(\alpha)^{pq} = \mathcal{R}(\alpha)^{pq}$ for any $\alpha\in \bigoplus\limits_{x, y\in I}V_{yd}^e\otimes V_{xc}^y\otimes V_{ab}^x$ and $p, q\in I$.
\end{proof}

The equation in the statement of the theorem \ref{Theorem:ShortForm} can be seen as an alternative short form of the pentagon relation. It's clear that these two forms are equivalent.

In the literature (see, for example, \cite{Ca, Tu}) the components $\mathcal{F}_{abc}^{d}|^{x}_{y}$ of the map $\mathcal{F}_{abc}^{d}$ are called 6j-symbols and denoted by $$\left\{
    \begin{array}{ccc}
        a & b & x \\
        c & d & y
    \end{array}
\right\}\colon V_{xc}^{d}\otimes V_{ab}^{x}\to V_{ay}^{d}\otimes V_{bc}^{y}.$$ The equation from the theorem \ref{Theorem:ShortForm} is often called the Biedenharn -- Elliott identity. 

\subsection{Tensor form}

Fix a base $\{e_i\}_{i = 1}^{n}$ of the module $V_{xc}^d$, $\{f_j\}_{j = 1}^{m}$ of the module $V_{ab}^x$, $\{g_r\}_{r = 1}^{k}$ of the module $V_{ay}^d$ and $\{h_s\}_{s = 1}^{l}$ of the module $V_{bc}^y$. In these bases the map $\mathcal{F}_{abc}^d|^{x}_{y}\colon V_{xc}^d\otimes V_{ab}^x \to V_{ay}^d\otimes V_{bc}^y$ is defined by constants $R_{ij}^{rs}\in \K$ such that $$\mathcal{F}_{abc}^d|^{x}_{y}(e_i\otimes f_j) = \sum\limits_{s, r}R_{ij}^{rs} g_r\otimes h_s.$$

Define the tensor $$\mathbb{F}_{abc}^{d}|^{x}_{y}\colon V_{xc}^d\otimes V_{ab}^x\otimes (V_{ay}^{d})^* \otimes (V_{bc}^{y})^* \to \K$$ by the same coordinates $R_{ij}^{rs}$: $$\mathbb{F}_{abc}^{d}|^{x}_{y}(e_i\otimes f_j\otimes g^{r}\otimes h^{s}) = R_{ij}^{rs},$$ where $\{g^r\}_{r = 1}^{k}$ is a dual basis of $(V_{ay}^{d})^*$ and $\{h^s\}_{s = 1}^{l}$ is a dual basis of $(V_{bc}^{y})^*$.

Let $T\colon U\otimes U^{*}\otimes W\to \K$ be a tensor. Let $\{e_i\}_{i = 1}^{n}$ be a basis of $U$, $\{e^i\}_{i = 1}^{n}$ be a dual basis of $U^{*}$ and $\{f_j\}_{j = 1}^{m}$ be a basis of $W$. Then define the result of the contraction of the tensor $T$ along the pair of modules $U$ and $U^{*}$ (denote it $*_{U}T\colon W\to\K$) by the formula $$*_{U}T(f_j) = \sum\limits_{i = 1}^{n}T(e_i\otimes e^i\otimes f_j).$$

It's clear that contractions along different pairs of dual modules are commutes.

For the family of tensors $\{\mathbb{F}_{abc}^{d}|^{x}_{y}\}$ denote the construction along the pair of modules $V_{ab}^{c}$ and $(V_{ab}^{c})^*$ by $*_{ab}^{c}$.

\begin{theorem}
    \label{Theorem:Tensors}
    For any $a, b, c, d, e, x, y, p, q\in I$: $$\sum\limits_{z\in I}*_{zd}^{p} *_{bc}^{z} *_{az}^{y}(\mathbb{F}_{bcd}^{p}|^{z}_{q}\otimes \mathbb{F}_{azd}^{e}|^{y}_{p}\otimes \mathbb{F}_{abc}^{y}|^{x}_{z}) = *_{xq}^{e}(\mathbb{F}_{xcd}^{e}|^{y}_{q}\otimes \mathbb{F}_{abq}^{e}|^{x}_{p}).$$
\end{theorem}
\begin{proof}
    We will use the approach from the proof \cite[Corollary VI.1.5.3]{Tu}. For each map $$\mathcal{F}_{abc}^{d}|^{x}_{y}\colon V_{xc}^{d}\otimes V_{ab}^{x}\to V_{ay}^{d}\otimes V_{bc}^{y}$$ consider the vector $$F_{abc}^{d}|^{x}_{y}\in (V_{xc}^{d})^*\otimes (V_{ab}^{x})^*\otimes V_{ay}^{d}\otimes V_{bc}^{y},$$ defined by the formula $$F_{abc}^{d}|^{x}_{y} = \sum\limits_{i, j, r, s} R_{ij}^{rs}\cdot e^i\otimes f^j\otimes g_r\otimes h_s,$$ where as previously $\{e^i\}_{i = 1}^{n}$ is a basis of $(V_{xc}^{d})^*$ dual to the basis $\{e_i\}_{i = 1}^{n}$ of the $V_{xc}^{d}$, $\{f^j\}_{j = 1}^{m}$ is a basis of $(V_{ab}^{x})^*$ dual to the basis $\{f_j\}_{j = 1}^{m}$ of the $V_{ab}^{x}$, $\{g_r\}_{r = 1}^{k}$ is a basis of $V_{ay}^{d}$, $\{h_s\}_{s = 1}^{l}$ is a basis of $V_{bc}^{y}$, and $$\mathcal{F}_{abc}^{d}|^{x}_{y}(e_i\otimes f_j) = \sum\limits_{r, s}R_{ij}^{rs}\cdot g_r\otimes h_s.$$

    For any vector $\alpha\in U\otimes U^*\otimes W$ define the result of contraction $*_{U}(\alpha)$ along the pair of modules $U$ and $U^{*}$ as follows. If $\{e_i\}_{i = 1}^{n}$ is a basis of $U$, $\{e^i\}_{i = 1}^{n}$ is a dual basis of $U^*$, $\{f_j\}_{j = 1}^{m}$ is a basis of $W$, and $\alpha = \sum\limits_{i, j, k}\alpha^{ik}_{j} e_i\otimes e^j\otimes f_k$, then $$*_{U}(\alpha) = \sum\limits_{i, k}\alpha^{ik}_{i} f_k\in W.$$

    Denote contraction along the pair $V_{ab}^c$ and $(V_{ab}^{c})^{*}$ by $*_{ab}^c$. Notice that $$\mathcal{F}_{abc}^{d}|^{x}_{y}(\alpha) = *_{xc}^{d} *_{ab}^{x} (F_{abc}^{d}|^{x}_{y}\otimes \alpha).$$

    Let $\mathcal{L}$ be the left hand side of the statement in the theorem \ref{Theorem:ShortForm}, and let $\mathcal{R}$ be the right hand side of that statement. Then for any $\alpha\in V_{yd}^{e}\otimes V_{xc}^{y}\otimes V_{ab}^{x}$: $$\mathcal{L}(\alpha) = \sum\limits_{z\in I} *_{zd}^{p} *_{bc}^{z} *_{yd}^{e} *_{az}^{y} *_{xc}^{y} *_{ab}^{x} (F_{bcd}^{p}|^{z}_{q}\otimes F_{azd}^{e}|^{y}_{p}\otimes F_{abc}^{y}|^{x}_{z}\otimes \alpha).$$

    Simillary $$\mathcal{R}(\alpha) = *_{xq}^{e} *_{ab}^{x} *_{yd}^{e} *_{xc}^{y} (F_{abq}^{e}|^{x}_{p}\otimes F_{xcd}^{e}|^{y}_{q}\otimes \alpha).$$

    Theorem \ref{Theorem:ShortForm} states that $\mathcal{L}(\alpha) = \mathcal{R}(\alpha)$ for any $\alpha\in V_{yd}^{e}\otimes V_{xc}^{y}\otimes V_{ab}^{x}$. Then $$\sum\limits_{z\in I} *_{zd}^{p} *_{bc}^{z} *_{az}^{y} (F_{bcd}^{p}|^{z}_{q}\otimes F_{azd}^{e}|^{y}_{p}\otimes F_{abc}^{y}|^{x}_{z}) = *_{xq}^{e} (F_{abq}^{e}|^{x}_{p}\otimes F_{xcd}^{e}|^{y}_{q}).$$

    The statement of the theorem \ref{Theorem:Tensors} follows from two facts: for any $a, b, c, d, x, y\in I$, the coordinates of the tensor $\mathbb{F}_{abc}^{d}|^{x}_{y}$ and the vector $F_{abc}^{d}|^{x}_{y}$ in a fixed bases are the same, and under contraction they change in the same way.
\end{proof}

In \cite[Chapter VI]{Tu} tensors $\mathbb{F}_{abc}^{d}|^{x}_{y}$ from the statement of the theorem \ref{Theorem:Tensors} and vectors $F_{abc}^{d}|^{x}_{y}$ from the proof are denoted by 
\begin{center}
    $\left\{
        \begin{array}{ccc}
            a & b & x \\
            c & d & y
        \end{array}
    \right\}''$ and $\left\{
        \begin{array}{ccc}
            a & b & x \\
            c & d & y
        \end{array}
    \right\}'$
\end{center}
respectively.

\subsection{Normalized tensor form}

For each $x\in I$, introduce a parameter $w_x\in\K$. Then denote $$\left|\begin{array}{ccc}
    a & b & x \\
    c & d & y
\end{array}\right| = \frac{1}{w_y}\cdot \mathbb{F}_{abc}^{d}|^{x}_{y}.$$

\begin{theorem}
    \label{Theorem:Symbols}
    For any $a, b, c, d, e, x, y, p, q\in I$: 
    \begin{multline*}
        \sum\limits_{z\in I}w_z\cdot *_{zd}^{p} *_{bc}^{z} *_{az}^{y} \left(\left|\begin{array}{ccc}
            b & c & z \\
            d & p & q
        \end{array}\right|\otimes \left|\begin{array}{ccc}
            a & z & y \\
            d & e & p
        \end{array}\right|\otimes \left|\begin{array}{ccc}
            a & b & x \\
            q & e & p
        \end{array}\right|\right) = \\ = *_{xq}^{e}\left(\left|\begin{array}{ccc}
            x & c & y \\
            d & e & q
        \end{array}\right|\otimes \left|\begin{array}{ccc}
            a & b & x \\
            q & e & p
        \end{array}\right|\right).
    \end{multline*}
\end{theorem}
\begin{proof}
    It follows from the theorem \ref{Theorem:Tensors} by substituting 
    \begin{center}
        $\mathbb{F}_{bcd}^{p}|^{z}_{q} = w_q \cdot \left|\begin{array}{ccc}
            b & c & z \\
            d & p & q \\
        \end{array}\right|$, 
        $\mathbb{F}_{azd}^{e}|^{y}_{p} = w_p \cdot \left|\begin{array}{ccc}
            a & z & y \\
            d & e & p \\
       \end{array}\right|$, 
       $\mathbb{F}_{abc}^{y}|^{x}_{z} = w_z \cdot \left|\begin{array}{ccc}
           a & b & x \\
           c & y & z \\
        \end{array}\right|$,

        $\mathbb{F}_{xcd}^{e}|^{y}_{q} = w_q \cdot \left|\begin{array}{ccc}
            x & c & y \\
            d & e & q \\
        \end{array}\right|$, 
        $\mathbb{F}_{abq}^{e}|^{x}_{p} = w_p \cdot \left|\begin{array}{ccc}
            a & b & x \\
            q & e & p \\
        \end{array}\right|$.
    \end{center}
\end{proof}

The family of tensors $\left\{\left|\begin{array}{ccc}
    a & b & x \\
    c & d & y
\end{array}\right|\right\}$ plays an important role in the theory of invariants for 3-manifolds. If for each $a, b, c, d, x, y\in I$ the tensor $\left|\begin{array}{ccc}
    a & b & x \\
    c & d & y
\end{array}\right|$ is symmetric in the sense that $$\left|\begin{array}{ccc}
    a & b & x \\
    c & d & y
\end{array}\right| = \left|\begin{array}{ccc}
    b & a & x \\
    d & c & y
\end{array}\right| = \left|\begin{array}{ccc}
    y & d & a \\
    x & b & c
\end{array}\right|,$$ then these tensors and values $w_x\in \K$, $x\in I$, define a Turaev -- Viro type invariant for closed 3-manifolds. The statement of the theorem \ref{Theorem:Symbols} is also often referred to as the Biedenharn -- Elliott identity.

\section{Solution of the pentagon relation from modular categories}

Let $\mathcal{C}$ be a modular category. To construct the solution of the pentagon relation, we should fix the set of colours $I$ and modules $V_{ab}^{c}$ for all $a, b, c\in I$.

Let the set $I$ be a set of simple objects of the category $\mathcal{C}$.

For any three simple objects $a, b, c\in I$ define $V_{ab}^c = Hom(c\to b\otimes a)$, i.e. the module $V_{ab}^{c}$ is a set of morphisms from the object $c$ to the object $b\otimes a$. This set is equipped by a natural module structure, because the category $\mathcal{C}$ is modular.

For any $a, b, c, d\in I$ consider two isomorphisms $$A_{abc}^d\colon \bigoplus_{x\in I}V_{xc}^d\otimes V_{ab}^{x}\to Hom(d\to c\otimes (b\otimes a)),$$ $$B_{abc}^{d}\colon\bigoplus_{y\in I} V_{ay}^{d}\otimes V_{bc}^{y}\to Hom(d\to (c\otimes b)\otimes a),$$ defined as follows. The isomorphism $A_{abc}^{d} = \bigoplus\limits_{x\in I}A_{abc}^{d}|^{x}$, where the homomorphism $$A_{abc}^{d}|^{x}\colon V_{xc}^d\otimes V_{ab}^{x}\to Hom(d\to c\otimes (b\otimes a))$$ is defined by the formula $$A_{abc}^{d}|^{x}(f_x\otimes g_x) = f_x\circ (id_c\otimes g_x).$$ As previously we write the morphisms composition from left to right.

Similarly, $B_{abc}^{d} = \bigoplus\limits_{y\in I}B_{abc}^d|^{y}$, where the homomorphism $$B_{abc}^{d}|^{y}\colon V_{ay}^d\otimes V_{bc}^{y}\to Hom(d\to (c\otimes b)\otimes a)$$ is defined by the formula $$B_{abc}^{d}|^{y}(f_y\otimes g_y) = f_y\circ (g_y\otimes id_a).$$

Finally, $$\mathcal{F}_{abc}^{d} = A_{abc}^{d}\circ \alpha_{c, b, a}^{-1}\circ (B_{abc}^{d})^{-1},$$ where $\alpha_{c, b, a}\colon (c\otimes b)\otimes a \to c\otimes (b\otimes a)$ is an associativity isomorphism which defines the monoidal structure on the category $\mathcal{C}$.

In \cite[Theorem VI.1.5.1]{Tu} it is proved that the family of maps $\{\mathcal{F}_{abc}^{d}\}$ satisfies the pentagon relation.

\end{document}